\documentclass{amsart}
\usepackage{float}
\usepackage{subfiles}
\usepackage{amsmath,amsfonts,amscd,amssymb,color}
\usepackage{float}
\usepackage{hyperref}
\usepackage{layout}
\usepackage{graphicx}
\usepackage{cite}
\usepackage{color}
\usepackage{appendix}
\usepackage{amssymb,latexsym,amsthm}
\usepackage{amsfonts}
\usepackage{mathtools}
\usepackage{enumitem}
\usepackage{multicol}
\usepackage{multirow}
\SetLabelAlign{CenterWithParen}{(\makebox[0.8em]{#1})}

\usepackage{xcolor}

\usepackage{nomencl}
\makenomenclature

\usepackage{tikz-cd}
\usetikzlibrary{matrix, calc, arrows}
\usetikzlibrary{arrows.meta}

\def\newmathop#1{\expandafter\gdef\csname #1\endcsname{\mathop{\rm #1}\nolimits}}
\def\newvmathop#1{\expandafter\gdef\csname v#1\endcsname{\mathop{\rm #1}\nolimits}}

\theoremstyle{plain}
\newcounter{thmcount}[section]

\newtheorem{theorem}[thmcount]{Theorem}

\newtheorem{lemma}[thmcount]{Lemma}

\newtheorem*{conjecture*}{Conjecture}
\newtheorem*{theorem*}{Theorem}
\newtheorem*{corollary*}{Corollary}
\newtheorem*{lemma*}{Lemma}

\newtheorem*{rough*}{Rough idea}

\newtheorem*{motivation*}{Motivation}

\newtheorem*{goal*}{Goal}

\theoremstyle{definition}
\newtheorem{remark}[thmcount]{Remark}
\newtheorem{example}[thmcount]{Example}
\newtheorem*{example*}{Example}
\newtheorem*{remark*}{Remark}
\newtheorem{definition}[thmcount]{Definition}

\numberwithin{equation}{section}

\def\Q{{\mathbb Q}}
\def\F{{\mathbb F}}

\def\Fpb{{\overline{\mathbb F}_p}}

\def\Qp{{{\mathbb Q}_p}}
\def\Qb{\overline{\mathbb Q}}

\def\Qpb{{\overline{\mathbb Q}_p}}
\def\Z{{\mathbb Z}}

\def\C{{\mathbb C}}

\newmathop{Disc}
\newmathop{Tr}
\newmathop{Norm}
\newmathop{ord}
\newmathop{GL}
\newmathop{SL}
\newmathop{PGL}
\newmathop{PSL}
\newmathop{Hom}
\newmathop{Ind}
\newmathop{Res}
\newmathop{rk}
\newmathop{corank}
\newmathop{coker}
\newmathop{codim}
\newmathop{cyc}
\newmathop{Reg}
\newmathop{Gal}
\newmathop{Sel}

\newmathop{BSD}
\newmathop{tors}
\newmathop{tr}
\newmathop{Ext}
\newmathop{Mod}

\newmathop{id}
\newmathop{Aut}
\newmathop{Art}

\newmathop{Rep}
\DeclareMathOperator{\HT}{HT}

\DeclareMathOperator{\Ha}{Ha}

\renewcommand{\det}{\operatorname{det}}

\newcommand{\mtwo}[4]{\begin{pmatrix} #1 & #2  \\ #3 & #4 \end{pmatrix}}
\newcommand{\mtwosmall}[4]{\begin{psmallmatrix} #1 & #2  \\ #3 & #4 \end{psmallmatrix}}
\newcommand{\tta}[1]{\tilde{#1}}

\newcommand{\Frob}{\text{Fr}}
\newcommand{\ftau}[1]{\Frob^{#1} \circ \tau}

\title[Partial weight one modularity]{Partial weight one modularity for Galois representations associated to mod $p$ Hilbert modular forms}
\author{Hanneke Wiersema}
\date{\today}
\email{hanneke.wiersema@maths.ox.ac.uk}

\begin{document}
\maketitle

\begin{abstract}
Let $p$ be an odd prime. Let $\rho: G_F \to \GL_2(\Fpb)$ be a Galois representation of a totally real field $F$. For a small partial weight one weight $(k,0)$, we prove that modularity of $\rho$ can be characterised using $p$-adic Hodge theory, as conjectured by Diamond and Sasaki. We show that if $\rho$ is modular with respect to a partial weight one mod $p$ Hilbert modular form, then each of its local representations has a crystalline lift with prescribed Hodge--Tate weights. Conversely, if for each $v|p$ the restriction $\rho|_{G_{F_v}}$ has a crystalline lift with certain irregular weights, we show that $\rho$ arises from a partial weight one Hilbert modular form. Our method consists of translating results from regular to irregular weights. We do this globally, relating modularity of regular weights to modularity of irregular weights and vice versa, and also use the local, $p$-adic Hodge theory analogue of this, which is recent work of the author.
\end{abstract}

\section{Introduction}
Let $\rho: \Gal(\overline{F}/F) \to \GL_2(\Fpb)$ be a mod $p$ Galois representation, where $F \neq \Q$ is a totally real field. If $\rho$ is irreducible, continuous and totally odd, then $\rho$ is believed to be modular in the sense that it is equivalent to the reduction of the Galois representation of a characteristic zero Hilbert modular form. We will consider modularity in characteristic $p$, instead saying $\rho$ is modular if it is equivalent to the mod $p$ Galois representation attached to a mod $p$ Hilbert modular form. These viewpoints differ, as not every mod $p$ Hilbert modular form can be lifted to one in characteristic zero. 

Consider this when $F=\Q$. In this case, Serre predicted modularity of Galois representations $\rho$ as above by considering mod $p$ Galois representations associated to modular forms in characteristic zero \cite{serreduke}. This prediction is now known due to Khare and Wintenberger \cite{kharewintenberger}, \cite{khare}, building on work of many people including Edixhoven, who, in contrast to Serre, studied modularity using geometric mod $p$ modular forms\cite{Edixhoven1992}.

In the classical setting these two viewpoints are equivalent for all modular forms with weight $k \geq 2$, as such forms lift from characteristic $p$ to characteristic zero. The only difference concerns Galois representations corresponding to weight one modular forms. The mod $p$ Galois representations that come from (geometric) mod $p$ weight one modular forms are exactly the representations that are unramified at $p$ (see \cite{Edixhoven1992}). In the Hilbert modular form setting this becomes more complex, and it is on this complexity that we focus, inspired by recent work of Diamond and Sasaki \cite{DS}. Throughout we will assume that $p$ is an odd prime and that $p$ is unramified in $F$ (note the recent work of Diamond and Sasaki \cite{DSram} in the ramified case). 

In this paper we study Galois representations expected to correspond to partial weight one mod $p$ Hilbert modular forms, which are forms that do not necessarily lift to characteristic zero. Note that such Galois representations also need not be unramified at places above $p$. This makes it less clear how to characterise the modularity of Galois representations coming from partial weight one forms. Diamond and Sasaki provide a conjectural answer to this question which suggests that one can use $p$-adic Hodge theory to do this, as long as the weights lie in a certain minimal weight cone. More specifically, a Galois representation should arise from a mod $p$ Hilbert modular form of weight $(k,l)$ in this cone if and only if for each $v|p$ the restriction $\rho|_{G_{F_v}}$ has a crystalline lift with Hodge--Tate weights $(k_{\tau},l_{\tau})_{\tau \in \Sigma_v}$ \cite[Conjecture 7.3.2]{DS}. This statement includes the partial weight one weights in the minimal weight cone, which are examples of irregular weights.  

We say a weight $(k,l)$ is \emph{regular} if $k_{\tau} \geq 2$ for all embeddings $\tau: F \to \Qb$, otherwise the weight is \emph{irregular}. For regular weights the connection between modularity and $p$-adic Hodge theory is apparent in the work of Buzzard, Diamond and Jarvis \cite{bdj}. They define a notion of modularity with respect to mod $p$ representations of $\GL_2(\mathcal{O}_F/p)$ in terms of $\rho$ appearing in the cohomology of certain Shimura curves over $F$. This reinterpretation of modularity is commonly referred to as \emph{algebraic} modularity.

For a fixed $\rho$, Buzzard, Diamond and Jarvis provide a conjectural weight recipe in this setting using $p$-adic Hodge theory, motivated by the study of the reductions of crystalline representations. The conjecture is now proven under some technical hypotheses in \cite{geekisin},\cite{GLS14}, and \cite{newton}.  Diamond and Sasaki conjecture that \emph{geometric} modularity, i.e. being modular with respect to a (geometric) mod $p$ Hilbert modular form, and algebraic modularity are equivalent for regular weights \cite[Conjecture 7.5.2]{DS}. There are  results towards this due to Diamond and Sasaki, who prove the direction from algebraic to geometric \cite{DSram}, and Yang \cite{yang}, which we describe in more detail in Section \ref{mainresultsection}. We emphasise that the notion of algebraic modularity is not even defined for irregular weights.

Let us say a little about what is known about irregular weights.  If $\rho$ corresponds to a mod $p$ Hilbert modular form of parallel weight one, then by work of \cite{dimitrovwiese} (and also \cite{emertonreduzzixiao}), we know that the Galois representation is unramified at all $v|p$, under some mild hypotheses. Moreover, by recent work of De Maria \cite{demaria}, if we have a mod $p$ Hilbert modular form of paritious weight where the weight is parallel one above some place $v|p$, we know that then indeed the Galois representation will be unramified at this place. Conversely, suppose $\rho: G_F \to \GL_2(\Fpb)$ is unramified at all places above $p$, which implies each of the restrictions $\rho|_{G_{F_v}}$ has a crystalline lift of parallel weight one, then, by work of Gee and Kassaei \cite{geekassaei}, under some hypotheses, we know that there exists a mod $p$ Hilbert modular form of parallel weight one whose Galois representation is equivalent to $\rho$.

We do not consider these parallel weight one cases, but focus on partial weight one. The results mentioned above are specific to parallel weight one and as such do not have an easy extension to partial weight one. Our main result shows that for partial weight one weights indeed one can characterise the modularity of $\rho$ in terms of the existence of crystalline lifts:

\begin{theorem}[Theorem \ref{locglobviceversa}]
\label{thmintro}
Let $F$ be a totally real field in which $p$ is unramified, and suppose $p$ is odd. Suppose that $\rho:G_F \to \GL_2(\Fpb)$ is irreducible and modular. Suppose that for $\rho$ the Buzzard--Diamond--Jarvis conjecture holds and that for regular weights algebraic and geometric modularity are equivalent. Let $(k,0)$ be an irregular weight satisfying the conditions as in Theorem \ref{locglobviceversa}.

Then $\rho$ is geometrically modular of weight $(k,0)$ if and only if $\rho|_{G_{F_v}}$ has a crystalline lift of weight $(k_{\tau},0)_{\tau \in \Sigma_v}$ for all $v \mid p$.
\end{theorem}

The direction from geometric modularity to crystalline liftability can be viewed as a local-global compatibility statement, and this is also the subject of upcoming work of Kansal, Levin and Savitt \cite{kls}. We remark that their methods are more geometric than the ones present in this paper and allow for a larger weight range. For the other direction, the result allows one to obtain the \emph{existence} of a partial weight one Hilbert modular form whose Galois representation is equivalent to $\rho$. Under the assumptions listed the theorem then proves part of the conjecture of Diamond--Sasaki alluded to above. This extends work of Diamond and Sasaki in the case where $F$ is quadratic and $p$ is inert (\cite[Theorem 11.4.1 and Theorem 11.4.3]{DS}). Regarding the assumptions in our result, recall from above that the Buzzard--Diamond--Jarvis conjecture is known under some hypotheses and that there is significant recent progress on the relation between algebraic and geometric modularity as we will discuss in Section \ref{mainresultsection}. 

Let us next discuss our method of proof. This consists of translating results from regular to irregular weights.  In this paper we do this globally, relating modularity of regular weights to geometric modularity of irregular weights and vice versa, generalising a result in the inert quadratic case by Diamond and Sasaki \cite[Lemma 11.3.1]{DS}. This is the first theorem in this paper (Theorem \ref{mainthm}). The main work in generalising their result lies in finding explicit formulae for the weights. One needs to, given any irregular weight, find a set of corresponding regular weights whose modularity will imply modularity of the irregular weight (under some technical conditions) which we do in Section \ref{weightcombinatorics}. 

The proof uses methods introduced in \cite{DS} involving partial Hasse invariants and partial Theta operators. These weight combinatorics inspired the local $p$-adic Hodge theory analogue of this, proven in \cite[Theorem 1.2]{hwcrys}. The local version means that one can describe crystalline liftability of certain irregular weights in terms of the crystalline liftability of multiple regular weights. 
\par
The proof of Theorem \ref{thmintro} is then essentially a series of equivalences: above each place $v|p$, crystalline liftability of the irregular weight is equivalent to crystalline liftability of the regular weights. This existence of such crystalline lifts above each place is, for our weights, equivalent to algebraic modularity of the same regular weights due to work of Gee, Liu and Savitt \cite{GLS15}. By our assumptions this is equivalent to geometric modularity of the same regular weights. Finally this is equivalent to geometric modularity of the irregular weight due to Theorem \ref{mainthm}. As we shall see, in order to prove these results we need to impose some assumptions on $\rho$ and on our weights, which come from Theorem \ref{mainthm} in this paper as well as conditions appearing in the local version of this result in \cite{hwcrys}.

\subsection*{Outline of the paper}\label{ss:outline} We first introduce some notation and terminology. In the next section we briefly discuss mod $p$ Hilbert modular forms, weight shifting operators and geometric modularity. In the third section we construct the regular weights that will feature in our weight transferring results and we finish this section by proving the theorem relating geometric modularity of irregular weights to that of regular weights. In the fourth and final section we state and prove the main theorem of this paper.

\subsection*{Notation}\label{ss:notation}
We let $p$ be an odd rational prime. Fix algebraic closures $\Qpb$ and $\overline{\Q}$ of $\Qp$ and $\Q$ respectively, and fix embeddings of $\overline{\Q}$ into $\Qpb$ and $\C$. Let $F$ be a totally real field in which $p$ is unramified.  We let $\mathcal{O}_F$ be the ring of integers of $F$ and $\mathcal{O}_{F,p}=\mathcal{O}_F \otimes \Z_p$. We let $\hat{\mathcal{O}}_F$ denote the profinite completion of $\mathcal{O}_F$. Let $\Sigma$ be the set of embeddings $F \to \overline{\Q}$.

We let $L/\Q_p$ be a sufficiently large finite extension of $\Qp$ containing the image of every embedding in $\Sigma$. Let $E$ be the residue field. We identify $\Sigma$ with the set of embeddings of $F$ into $L$ and with the set of embeddings of $\mathcal{O}_F/p$ into $\Fpb$. We let $\Frob$ be the absolute Frobenius of $\Fpb$. For $v \mid p$ we denote $\Sigma_v$ for the set of embeddings of $\mathcal{O}_F/v$ into $\Fpb$. Note that Frobenius has a cyclic action on $\Sigma_v$. For any $\tau_0 \in \Sigma_v$ we let $\tau_i=\tau_{i+1}^p$ so that $\tau_i= \Frob \circ \tau_{i+1}$.  We often identify $\Sigma$ with $\prod_{v \mid p} \Sigma_v$. For any $\tau \in \Sigma$, we let $e_{\tau}$ denote the canonical basis element of $\Z^{\Sigma}$ associated to $\tau$.

Let $K$ be any unramified extension of $\Q_p$ of degree $f$ with residue field $\F$. Write $\Sigma_{K}$ for the embeddings of $K$ into $\Qpb$. We identify $\Sigma_{K}$ with the set of embeddings of $\F \to \Fpb$. For any $\tau_0 \in \Sigma_{K}$ we let $\tau_i=\tau_{i+1}^p$ so that $\Sigma_{K}=\{ \tau_i \, \vert \, i \in \{ 0, \dots, f-1\} \}$. We write $\omega_{\tau}$ for the fundamental character corresponding to $\tau \in \Sigma_K$.

\subsubsection*{$p$-adic Hodge theory}

If $W$ is a de Rham representation of $G_K$ over $\Qpb$ and $\kappa$ is an embedding $K \hookrightarrow \Qpb$ then the multiset $\HT_{\kappa}(W)$ of Hodge--Tate weights of $W$ with respect to $\kappa$ is defined to contain the integer $i$ with multiplicity 
\[
\dim_{\Qpb}(W \otimes_{\kappa,K} \hat{\overline{K}}(i))^{G_K},
\]
with $\hat{\overline{K}}$ the completion of $\overline{K}$. We wish to emphasise that with this convention the $p$-adic cyclotomic character $\chi_{\text{cyc}}$ has Hodge--Tate weight $-1$ (as in \cite{DS}). 

The following is taken from \cite{DS}.

\begin{definition}
We say a representation $\rho: G_K \to \GL_2(\Fpb)$ has a crystalline lift with weights $(k,l) \in \Z_{\geq 1}^{\Sigma_{K}} \times \Z^{\Sigma_{K}}$ if there exists a continuous representation $\tilde{\rho}: G_K \to \GL_2(\mathcal{O}_L)$ such that $\tilde{\rho} \otimes_{\mathcal{O}_L} \F_L \cong \rho$ and $\tilde{\rho} \otimes_{\mathcal{O}_L} L$ is crystalline with Hodge--Tate weights $\{k_{\kappa}+l_{\kappa}-1,l_{\kappa}\}_{\kappa}$ where $\kappa \in \Hom(K,\Qpb)$.
\end{definition}

\section{Mod $p$ Hilbert modular forms, weight shifting operators and geometric modularity}
A crucial ingredient in all that follows is the ability to attach Galois representations to mod $p$ Hilbert modular forms. In this section we introduce the theory of geometric mod $p$ Hilbert modular forms and their Galois representations following Diamond and Sasaki \cite{DS}. 
We start by studying mod $p$ Hilbert modular forms and defining geometric modularity, spending some time on eigenforms before moving on to weight shifting operators. Finally we summarise some key results concerning both of these to refer to later. We will provide references for the sake of brevity.

\subsection{Mod $p$ Hilbert modular forms and their Galois representations}
\label{modforms}
Building on work of Andreatta and Goren \cite{andreattagoren}, Diamond and Sasaki define mod $p$ geometric Hilbert modular forms as sections of certain automorphic line bundles on Hilbert modular varieties. The weights of such forms are tuples $(k,l) \in \Z^\Sigma \times \Z^\Sigma$. If the weights are paritious, then in characteristic zero these sections are classical Hilbert modular forms. Diamond and Sasaki show that in characteristic $p$, the geometric Hilbert modular form can also be defined for non-paritious weights. 

For any weight $(k,l) \in \Z^\Sigma \times \Z^\Sigma$, we shall write $M_{k,l}(U;E)$ for the space consisting of all mod $p$ Hilbert modular forms $f$ of weight $(k,l)$ with coefficients in $E$, and of level $U$ (as in \cite[Def 3.2.2]{DS}). The level $U$ is an open compact subgroup of $\GL_2(\hat{\mathcal{O}}_F)$ containing $\GL_2(\mathcal{O}_{F,p})$, and we require $U$ to be sufficiently small (see \cite[Caveat 2.4.2]{DS}) and $p$-neat (see \cite[Def 3.2.3]{DS}). These conditions ensure that one can indeed define the appropriate line bundles for all $(k,l)$.

We introduce notation for particular subgroups that satisfy these conditions. If $\mathfrak{n}$ is an ideal of $\mathcal{O}_F$ we define $U_1(\mathfrak{n})$ as the open compact subgroup of $\GL_2(\hat{\mathcal{O}}_F)$ given by
\[
U_1(\mathfrak{n})=\left\{ \mtwo{a}{b}{c}{d} \in \GL_2(\hat{\mathcal{O}}_F) \, \vert \, c, d-1 \in \mathfrak{n} \hat{\mathcal{O}}_F \right\}.
\]

Let $v$ be a prime of $F$ such that $v \nmid p$. Write $S_v$ and $T_v$ for the Hecke operators on $M_{k,l}(U;E)$, defined as double coset operators (\cite[Section 4.3]{DS}). We will also need the Hecke operators $T_v$ for $v \mid p$ (when $l_{\tau}=0$ and $k_{\tau} \geq 2$ for all $\tau \in \Sigma$), whose construction is given in \cite[Section 9.7]{DS}. On the space $M_{k,0}(U_1(\mathfrak{n}),E)$ these commute with the $T_v$ for $v \nmid p$ and with each other, as well as with the $S_v$ for $v \nmid p \mathfrak{n}$ (assuming all $k_{\tau} \geq 2$).

We can now discuss Galois representations. Diamond and Sasaki attach representations to mod $p$ Hilbert eigenforms of \emph{arbitrary} weights, extending results by Goldring and Koskivirta \cite{goldringkoskivirta} and Emerton, Reduzzi and Xiao \cite{emertonreduzzixiao} for paritious weights.

Next let $Q$ be a finite set of primes containing all $v \mid p$ and all $v$ such that $\GL_2(\mathcal{O}_{F,v}) \not \subset U$. Suppose $f \in M_{k,l}(U;E)$ is an eigenform for $T_v$ and $S_v$ for all $v \not \in Q$, then we write $\rho_f$ for the Galois representation attached to $f$ as in \cite[Theorem 6.1.1]{DS}.

\begin{definition}
An irreducible, continuous and totally odd Galois representation $\rho: G_F \to \GL_2(\Fpb)$ is \emph{geometrically modular of weight $(k,l)$} if $\rho$ is equivalent to the extension of scalars of $\rho_f$ for some eigenform $f \in M_{k,l}(U;E)$ as above. 
\end{definition}

This means $\rho$ is geometrically modular of weight $(k,l)$ if there is a non-zero element $f \in M_{k,l}(U;E)$ for some $U \supset \GL_2(\mathcal{O}_{F,p})$ and $E \subset \Fpb$ such that
\[
T_v f=\tr(\rho(\Frob_v))f \quad \text{ and } \text{Nm}_{F/\Q}(v) S_v f = \det(\rho(\Frob_v))f,
\]
for all but finitely many primes $v$.

\subsubsection{Properties of Hilbert modular eigenforms}
In the following we will need to know more about the properties of the eigenforms that $\rho$ arises from. The aim is to define strongly stabilised eigenforms, which will play a crucial role in the proof of the main result in the next section. In order to do this, we will define twists of Hilbert modular forms by certain characters. One can use these characters to move between forms of weight $(k,l)$ and forms of weight $(k,l+l')$ for any $l'$ (in particular also $-l$). Details of what follows below can be found in \cite[Section 10]{DS}. In this subsection we assume $(k,l)$ is regular.

\subsubsection*{Twisted eigenforms}
Let $l' \in \Z^{\Sigma}$ and suppose $\mathfrak{m}, \mathfrak{n}$ are ideals in $\mathcal{O}_F$ such that $\mathfrak{m} \mid \mathfrak{n}$ and $\mathfrak{n}$ is prime to $p$. We let $V_{\mathfrak{m}} \subset \hat{\mathcal{O}}^{\times}_F$ denote the kernel of the natural projection to $(\mathcal{O}_F/\mathfrak{m})^{\times}$. We say a character
\[
\xi: \{ a \in (\mathbb{A}_F^{\infty})^{\times} \mid a_p \in \mathcal{O}_{F,p}^{\times}/V_{\mathfrak{m}}\} \to E^{\times},
\]
is a character of weight $l' \in \Z^{\Sigma}$ if $\xi(\alpha)=\overline{\alpha}^{l'}$ for all $\alpha \in F_{+}^{\times} \cap \mathcal{O}_{F,p}^{\times}$.

For any $f \in M_{k,l}(U_1(\mathfrak{n});E)$ and any character $\xi$ of weight $l'$ and conductor $\mathfrak{m}$, we write $f_{\xi} \in M_{k,l+l'}(U_1(\mathfrak{n}\mathfrak{m}^2);E)$ for the form defined in \cite[Section 10.3]{DS}. When this is non-zero and $f$ is an eigenform, then so is $f_{\xi}$. In this case we find that $\rho_{f_\xi} \cong \rho_{\xi'} \otimes \rho_f$, where $\xi': \mathbb{A}_F^{\times}/  F^{\times}F^{\times}_{\infty,+}V_{\mathfrak{m}p} \to E^{\times}$ is defined by:
\[
\xi'(\alpha z a)=\xi(a) \overline{a}_p^{-l'},
\]
where $\alpha \in F^{\times}, z \in F^{\times}_{\infty,+}, a \in (\mathbb{A}_F^{\infty})^{\times}$ and $a_p \in \mathcal{O}_{F,p}^{\times}$ (\cite[Lemma 10.3]{DS}). Here $\rho_{\xi'}: G_F \to E^{\times}$ is the character corresponding to $\xi'$ by class field theory.

\subsubsection*{Normalised eigenforms}
An eigenform $f \in M_{k,l}(U_1(\mathfrak{n});E)$ for Hecke operators $T_v$ for all $v \nmid p$ and $S_v$ for all $v \nmid \mathfrak{n}p$ is normalised if for all characters $\xi$ of weight $-l$ and conductor $\mathfrak{m}$ coprime to $p$, with values in extensions $E'$ of $E$, the twists $f_{\xi}$ are eigenforms for $T_v$ for all $v \mid p$. Finally, the coefficient of $q$ in the $q$-expansion should be $1$ (see \cite[Definition 10.5.1]{DS}).

If $f$ is a normalised eigenform in $M_{k,l}(U_1(\mathfrak{n}),E)$ and $\xi$ is as above then $f_{\xi}$ is a normalised eigenform in $M_{k,l+l'}(U_1(\mathfrak{n}\mathfrak{m}^2);E)$.

\subsubsection*{Stabilised eigenform}
A normalised eigenform $f \in M_{k,l}(U_1(\mathfrak{n});E)$ is \emph{stabilised} if $T_vf=0$ for all $v \mid \mathfrak{n}$. A stabilised eigenform is \emph{strongly stabilised} if and only if $T_v f_{\xi}=0$ for all $v \mid p$ and characters $\xi$ of weight $-l$.

The key lemma in the proof of our first result concerning these forms is the following uniqueness statement:
\begin{lemma}[{\cite[Lemma 10.6.5]{DS}}] 
\label{stronglystablemma}
There is at most one strongly stabilised eigenform $f \in M_{k,l}(U_1(\mathfrak{n});E)$ giving rise to $\rho$.
\end{lemma}

\subsection{Partial Hasse invariants and partial $\Theta$ operators}
\label{partoperators}
We next introduce some weight shifting operators on the space of mod $p$ Hilbert modular forms. These give us ways to increase the weight of a Hilbert modular form, which will allow us to move from irregular to regular weights.

\subsubsection*{Partial Hasse invariants}
For any $\tau \in \Sigma$, we let $\Ha_{\tau}$ denote the partial Hasse invariant as in   \cite[Section 5.1]{DS}. These are mod $p$ Hilbert modular forms of weight $k_{\Ha_{\tau}}:= p e_{\Frob^{-1} \circ \tau} - e_{\tau}$, so the weights of such forms are not strictly positive. Multiplication by $\Ha_{\tau}$ gives a new mod $p$ Hilbert modular form $\Ha_{\tau}f$, in fact, we obtain an injective map
\[
M_{k-k_{\Ha_{\tau}},l}(U;E) \to M_{k,l}(U;E),
\]
commuting with $T_v$ and $S_v$ for all $v \nmid p$ such that $\GL_2(\mathcal{O}_{F,v}) \subset U$ (with our usual assumptions on $U$, but see Proposition 5.1.1 of \cite{DS}). Moreover we have the following:

\begin{theorem}[{\cite[Theorem 1.1]{DiamondKassaei}.}]
Multiplication by $\Ha_{\tau}$ induces an isomorphism if $pk_{\tau}< k_{\Frob^{-1} \circ \tau}$. 
\label{hasseiso}
\end{theorem}

Moreover, Diamond and Kassaei deduce from this that every mod $p$ Hilbert modular form comes from multiplication by partial Hasse invariants from a form with weights in the following cone:
\[
\Xi_{\min} = \{ k \in \Z^{\Sigma} \, \vert \, pk_{\tau} \geq k_{\Frob^{-1} \circ \tau} \text{ for all } \tau \in \Sigma \},
\]
called the minimal weight cone or simply minimal cone (\cite[Corollary 1.2]{DiamondKassaei}).

\subsubsection*{Partial Theta operators}
For any $\tau \in \Sigma$, we let $\Theta_{\tau}$ denote the partial $\Theta$-operator. These are defined using the Hasse invariants and the Kodaira-Spencer isomorphism (\cite[Definition 8.2.1]{DS}). Let $(k,l) \in \Z^{\Sigma} \times \Z^{\Sigma}$ be any weight. Then we let $(\tta{k},\tta{l})$ be as follows: 
\begin{itemize}
\item if $\Frob \circ \tau=\tau$, then $\tta{k}_{\tau}=k_{\tau}+p+1$ and $\tta{k}_{\tau'}=k_{\tau'}$ if $\tau' \neq \tau$;
\item if $\Frob \circ \tau \neq \tau$, then $\tta{k}_{\tau}=k_{\tau}+1, \tta{k}_{\Frob^{-1} \circ \tau}=k_{\Frob^{-1} \circ \tau}+p$ and $\tta{k}_{\tau'}=k_{\tau'}$ if $\tau \not \in \{ \Frob^{-1} \circ \tau, \tau\}$;
\item $\tta{l}_{\tau}=l_{\tau}-1$, and $\tta{l}_{\tau'}=l_{\tau'}$ if $\tau' \neq \tau$.
\end{itemize}

The above represents the weight shift of the action of the $\Theta$-operator as follows: 

\begin{theorem}\cite[Theorem 8.2.2]{DS}
\label{thm822}
If $f \in M_{k,l}(U;E)$ and $\tau \in \Sigma$, then $\Theta_\tau(f) \in M_{\tta{k},\tta{l}}(U;E)$.  Moreover, $\Theta_\tau(f)$ is divisible by $\Ha_{\tau}$ if and only if either $f$ is divisible by $\Ha_{\tau}$ or $k_{\tau}$ is divisible by $p$.
\end{theorem}
We will also use the second part of the above theorem to prove divisibility by Hasse invariants.

\begin{remark} 
If $f$ is a normalised (resp. stabilised, strongly stabilised) eigenform, then this is also true for both $\Ha_{\tau}f$ and $\Theta_{\tau}(f)$ for any $\tau$ (assuming $k_{\tau} \geq 3$ if $\tau \neq \Frob \circ \tau$ in the case of $\Ha_{\tau}f$). This is \cite[Remark 10.6.6]{DS}.
\end{remark}

\subsection{Geometric modularity}
Finally we collect some results on geometric modularity we will need in the next section, all of these are results from \cite{DS}.

\begin{theorem}
\label{DSmain}
Suppose $\rho: G_F \to \GL_2(\Fpb)$ is irreducible and geometrically modular of weight $(k,l)$.
\begin{enumerate}
\item Then $\rho$ arises from an eigenform $f$ of weight $(k,l)$ and level $U_1(\mathfrak{n})$ for some $\mathfrak{n}$ prime to $p$.
\item Moreover, we can take $f$ such that $\Theta_{\tau}(f) \neq 0$ for all $\tau \in \Sigma$.
\item If $k_{\tau} \geq 2$ for all $\tau$, then $f$ can moreover be taken to be a normalised eigenform of level $U_1(\mathfrak{n})$ for some $\mathfrak{n}$ prime to $p$.
\item Moreover, $\rho$ is also geometrically modular of weight $(k+k_{\Ha_{\tau}},l)$ for any $\tau \in \Sigma$.
\item Moreover, $\rho$ is also geometrically modular of weight $(\tta{k},\tta{l})$ as in Section \ref{partoperators} for any $\tau \in \Sigma$.
\end{enumerate}
\end{theorem}

\begin{proof}
The first point is \cite[Lemma 10.2.1]{DS}, the second is \cite[Lemma 10.4.1]{DS} and the third is \cite[Proposition 10.5.2]{DS}.
The fourth follows from the commutativity of $\Ha_{\tau}$ with the Hecke operators. This is \cite[Theorem 10.4.2]{DS}, which follows from $(2)$ and commutativity of $\Theta_{\tau}$ with the Hecke operators. \qedhere
\end{proof}

\section{Partial weight one modularity}
\label{weightcombinatorics}
In this section we first prepare to state the theorem connecting modularity of regular weights to irregular weights and vice versa. Using the weight shifting operators introduced in the previous sections, we associate Hilbert modular forms of irregular weights to Hilbert modular forms with regular weights. We will then use the results in the previous section to prove our theorem.

\subsection{The regular weights}
\label{algweights}
Let $f$ be a partial weight one Hilbert modular form. We associate multiple Hilbert modular forms to it, all with regular weights. We obtain one form from $f$ through multiplication by Hasse invariants, and the other ones from multiplication by Hasse invariants and the action of a partial Theta operator. The idea behind this comes from Theorem \ref{DSmain}: if a Galois representation $\rho$ is modular with respect to the partial weight one weight form, it will also be modular with respect to the forms defined below.

Suppose $f$ is any non-zero mod $p$ Hilbert modular form with fixed irregular weight $(k,0)$ such that $k_{\tau} \geq 1$ for all $\tau \in \Sigma$. We first define some auxiliary subsets of $\Sigma$. Note that these depend on the choice of $k$, and not on $l$ or on the chosen $f$.

\begin{definition}
Let $k \in \Z_{\geq 1}^{\Sigma}$. We let $M$ be the subset of $\Sigma$ consisting of $\tau \in \Sigma$ such that
\[
 k_{\Frob^{-1} \circ \tau}=\dots =k_{\Frob^{-(s-1)} \circ \tau}=2 \text{ and } k_{\Frob^{-s} \circ \tau}=1,
\]
for some $s \geq 1$ (the first condition is vacuous if $s=1$).
\end{definition}

We need another subset of $\Sigma$, which is in turn a subset of $M$.
\begin{definition}
Let $k \in \Z_{\geq 1}^{\Sigma}$. Let $\tilde{M}$ be the set of $\tau \in \Sigma$ such that one of the following conditions holds:
\begin{itemize}
\item  $k_{\tau} \geq 3$ and $\tau \in M$,
\item $k_{\tau}=2, k_{\Frob^{-1} \circ \tau}=1, k_{\Frob^{-2} \circ \tau}=1$ or $2$ and if $k_{\Frob^{-2} \circ \tau}=2$ then $\Frob^{-2} \circ \tau \in M$.
\end{itemize}
\end{definition}

\begin{remark}
Note that if there is no case where $k_{\tau}=2$ and $k_{\Frob^{-1} \circ \tau}=1$, that $\tilde{M}$ is just the set $\tau \in \Sigma$ such that $k_{\tau} \neq 1, k_{\Frob^{-1} \circ \tau}=1$. Otherwise, there might also be $\tau \in \tilde{M}$ such that
\[
(k_{\tau},k_{\Frob^{-1} \circ \tau}, k_{\Frob^{-2} \circ \tau)}=(2,1,1) 
\]
or
\[
(k_{\tau},k_{\Frob^{-1} \circ \tau}, k_{\Frob^{-2} \circ \tau)}=(2,1,2) \text{ and } \Frob^{-2} \circ \tau \in M.
\]   
\end{remark}

\subsubsection*{The definition of $(k',l')$}
We let $k'$ be the weight of the Hilbert modular form obtained by multiplying $f$ by certain partial Hasse invariants. More precisely, we define $f'$ as the modular form: 
\begin{align}
\label{defkprime}
f \cdot \prod_{\tau \in M} \Ha_{\tau}. 
\end{align}
We denote $(k',l')$ for the weight of $f'$.  We have $l'=l=(0, \dots,0)$. Note again the values $(k',l')$ are independent on the choice of $f$.

\subsubsection*{The definition of $(k^{\mu},l^{\mu})$}
Next we want to obtain weights for each $\mu \in \tilde{M}$ as above. Again we do this by defining modular forms by acting on $f$. This time we use both partial Hasse invariants and partial Theta operators. 

Let $\mu \in \tilde{M}$, then we define $f^{\mu}$ to be the form given by
\begin{align}
\label{defkmu}
\Theta_\mu \left( f \cdot \prod_{\tau \in {M} \setminus \{\mu\}} \Ha_{\tau} \right).
\end{align}
We denote $(k^{\mu},l^{\mu})$ for the weight of $f^{\mu}$.  We have $l^{\mu}= - e_{\mu}$. In this case the values $(k^{\mu},l^{\mu})$ only depend on $k$ and the choice of $\mu$.

\begin{example}
Let $F$ be a cubic extension in which the prime $p$ is inert. Fix $\tau_0 \in \Sigma$ and consider the weight $(k,0) \in \Z_{\geq 1}^{\Sigma} \times \Z^{\Sigma}$ with $k=(k_{\tau_0},k_{\tau_1},k_{\tau_2})=(1,1,k_2)$ and $k_2 \geq 2$. 
We find that $M=\{\tau_0,\tau_2\}$ if $k_2 \geq 3$, and $M=\Sigma$ otherwise. We find $\tilde{M}=\{\tau_2\}$ in both cases.

By the above procedure we obtain weight tuples 
\begin{align*}
(k',l') = \begin{cases}
    ((p,p+1,k_2-1),(0,0,0)), & \text{if } k_2 \geq 3,  \\
    ((p,p,p+1),(0,0,0)), & \text{if } k_2=2, 
\end{cases}   
\end{align*}
and 
\begin{align*}  
(k^{\mu},l^{\mu}) = \begin{cases}
    ((p,p+1,k_2+1),(0,0,-1)), & \text{if } k_2 \geq 3,  \\
    ((p,p,p+3),(0,0,-1)), & \text{if } k_2=2. 
\end{cases}   
\end{align*}
\end{example}

Let us prove that indeed in this way we always obtain forms with regular weights.

\begin{lemma}
Let $f$ be a non-zero mod $p$ Hilbert modular form with irregular weight $(k,0)$ such that $k \in \Z_{\geq 1}^{\Sigma}$. Then the forms $f'$ and $f^{\mu}$ defined above have regular weight for each $\mu \in \tilde{M}$.
\end{lemma}
\begin{proof}
We fix an element $\mu \in \tilde{M}$. Then we have 
\[
k^{\mu}_{\mu}=
\begin{cases}
k_{\mu}+1, & \text{if } k_{\mu} \neq 2, \\ 
k_{\mu}+1+p, &  \text{if } k_{\mu} = 2, \\ 
\end{cases}
\text{ and }
k'_{\mu}=
\begin{cases}
k_{\mu}-1, &  \text{if }k_{\mu} \neq 2, \\ 
k_{\mu}-1+p, &  \text{if } k_{\mu} = 2, \\ 
\end{cases}
\]
Now note that we have
\[
(k'_{\tau},l'_{\tau})=(k^{\mu}_{\tau},l^{\mu}_{\tau})
\]
for all remaining embeddings, that is, for all $\tau \in \Sigma$ such that $\tau \neq \mu$.

So next consider $\tau \in \Sigma$ such that $\tau \neq \mu$. Then we have the following
\[
k'_{\tau}=k^{\mu}_{\tau}= 
\begin{cases}
    k_{\tau}, & k_{\tau} \geq 2, \tau \not \in M, \\
    k_{\tau}-1, &  k_{\tau} \geq 3, \tau \in M, \\
    k_{\tau}+p-1, &  k_{\tau} = 2, \tau \in M, \\
    p,  & k_{\tau}=1, \tau \in M, \\
    p+1, & k_{\tau}=1, \tau \not \in M,
\end{cases}
\]
so the result follows.
 \end{proof}

\subsection{Statement and proof of the theorem}
\label{statement}
We can now state the theorem and use all the results in the previous section to prove the main theorem.

\begin{theorem}
\label{mainthm}
Let $\rho: G_F \to \GL_2(\Fpb)$ be an irreducible representation. Let $(k,0) \in \Z_{\geq 1}^{\Sigma} \times \Z^{\Sigma}$ be an irregular weight such that $k \in \Xi_{\min}$. Suppose $\rho$ is geometrically modular of weight $(k,0)$, then 
\begin{enumerate}
\item $\rho$ is geometrically modular of weight $(k',0)$,
\item for each $\mu \in \tilde{M},$ $\rho$ is geometrically modular of weight $(k^{\mu},l^{\mu})$.
\end{enumerate}
On the other hand, if $(1)$ and $(2)$ hold, then, if we assume in addition that,
\begin{enumerate}[resume]
\item $\rho|_{G_{F_v}}$ is not of the form $\mtwosmall{\chi_1}{\ast}{0}{\chi_2}$
where $\chi_1|_{I_{F_v}}=\prod_{\tau \in \Sigma_v} \omega^{-l_{\tau}^{\mu}}_{\tau}$ and $\chi_2|_{I_{F_v}}=\prod_{\tau \in \Sigma_v} \omega^{1-k_{\tau}^{\mu}-l_{\tau}^{\mu}}_{\tau}$
for any $\mu \in \tilde{M}$ and any place $v \mid p$,
\item there is no place $v \mid p$ such that $k_{\tau}=1$ for all $\tau \in \Sigma_v$, 
\item $k_{\mu} \not \equiv 1 \mod p$ for all $\mu \in \tilde{M},$
\end{enumerate}
the converse holds, i.e. $\rho$ is geometrically modular of weight $(k,0)$.
\end{theorem}

We note first that the theorem holds just as well for $(k,l)$, with the regular weights again obtained from the formulae (\ref{defkprime}) and (\ref{defkmu}). We choose $l=0$ to ease exposition.

\begin{remark}
Condition (3) in our proof will ensure that certain Hilbert modular forms in our proof will be strongly stabilised so we can apply Lemma \ref{stronglystablemma}. We will then use this to obtain divisibility by the desired Hasse invariants.
\end{remark}

We exclude the weights in (4) as these would require methods similar to the aforementioned \cite{dimitrovwiese} and \cite{geekassaei}, and \cite{demaria}. The condition (5) is vital to our proof -- without it we could not obtain divisibility by Hasse invariants as a consequence of Theorem \ref{thm822}.

\begin{proof}
First assume $\rho$ is geometrically modular of $f \in M_{(k,0)}(U;E)$. By Theorem \ref{DSmain}(4) it is also geometrically modular of weight $(k',0)$ by multiplying $f$ by a product of partial Hasse invariants as in (\ref{defkprime}). For each embedding $\mu \in \tilde{M}$ we obtain forms $f^{\mu}$ of weight $(k^{\mu},l^{\mu})$ by multiplying by a product of partial Hasse invariants and acting on the result with partial Theta operators as in (\ref{defkmu}) in which case it follows from Theorem \ref{DSmain}(4-5) that $\rho$ is geometrically modular of these weights.
\par
For the converse, suppose (1)--(5) all hold. 
By (1) and Theorem \ref{DSmain}(3) $\rho$ arises from a normalised eigenform in $M_{{k'},{0}}(U_1(\mathfrak{m});E)$ for an ideal $\mathfrak{m}$ prime to $p$ (and sufficiently large $E$). 
From (2), and for any $\mu \in \tilde{M}$, Theorem \ref{DSmain}(3) again implies that $\rho$ arises from a normalised eigenform in $M_{{k}^{\mu},{l}^{\mu}}(U_1(\mathfrak{m}_{\mu});E)$ for some ideal $\mathfrak{m}_{\mu}$ prime to $p$.

Let $\mathfrak{n}=\mathfrak{m}^2  \prod_{\mu \in \tilde{M}} \mathfrak{m}_{\mu}^2$. Then $\mathfrak{n} \subset \mathfrak{m}$ satisfies the conditions in \cite[Lemma 10.6.2]{DS}, so this means $\rho$ arises from a stabilised eigenform $f' \in M_{k',0}(U_1(\mathfrak{n});E)$. Similarly, for each $\mu \in \tilde{M}$, the ideal $\mathfrak{n} \subset \mathfrak{m}_{\mu}$  satisfies the conditions in \cite[Lemma 10.6.2]{DS} with $\mathfrak{m}=\mathfrak{m}_{\mu}$, so this means $\rho$ arises from a stabilised eigenform $f^{\mu} \in M_{{k}^{\mu},{l}^{\mu}}(U_1(\mathfrak{n});E)$.

Fix $\mu \in \tilde{M}$. By \cite[Proposition 9.4.1]{DS} we can act on $f'$ with $\Theta_{\mu}$ to obtain a strongly stabilised eigenform $\Theta_{\mu}(f')$ of weight $(k'+p e_{\mu} + e_{\Frob^{-1} \circ \mu}, - e_{\mu})$ and level $\mathfrak{n}$. By assumption (3) and \cite[Corollary 10.7.2]{DS}, we find that $\Ha_{\mu}f^{\mu}$ is a strongly stabilised eigenform of weight $(k^{\mu}+p e_{\mu} - e_{\Frob^{-1} \circ \mu}, - e_{\mu})$ and level $\mathfrak{n}$. 

Now by Lemma \ref{stronglystablemma}, the modular form that $\rho$ arises from with these specifics is unique so that
\[
\Theta_{\mu}(f')=\Ha_{\mu}f^{\mu}.
\]
By Theorem \ref{thm822} and by assumption (5) we know that $f'$ is divisible by $\Ha_{\mu}$. 

We do this step for every $\mu \in \tilde{M}$ to find that $\rho$ is geometrically modular of a form $g$ such that
\begin{align}
\label{kg}
f'=g \prod_{\mu \in \tilde{M}} \Ha_{\mu}.
\end{align}
The form $g$ is of level $\mathfrak{n}$ and of weight $(k' - \sum_{\mu \in \tilde{M}}k_{\Ha_{\mu}},0)$. If this equals $(k,0)$, the proof is done.

If not, this means that the set $M':=M \setminus \tilde{M}$ is non-empty. We finish the proof by showing $g$ is divisible by $\prod_{\tau \in {M'}} \Ha_{\tau}$. We let $k'':=k' - \sum_{\mu \in \tilde{M}}k_{\Ha_{\mu}}$, so that $g$ has weight $(k'',0)$. Suppose $\tau \in M'$, then $\tau$ falls into one of the following three cases:
\begin{enumerate}[label=\roman*,align=CenterWithParen]
\item $k_{\tau}=1, \tau \in M$,
\item $k_{\tau}=2, k_{\Frob^{-1} \circ \tau}=1, k_{\Frob^{-2} \circ \tau} \neq 1$, and if $k_{\Frob^{-2} \circ \tau}=2$ then $\Frob^{-2} \circ \tau \not \in M$,
\item $k_{\tau}=2, k_{\Frob^{-1} \circ \tau}=2$ with $\tau \in M$.
\end{enumerate}

All of these cases follow from repeated applications of Theorem \ref{hasseiso}, which is the content of Lemma \ref{HasDiv}.
In the first case, for some $s \geq 1$ there must be an embedding $\Frob^{s} \circ \tau \in \tilde{M}$ such that $k_{\Frob^{s-1} \circ \tau}= \dots = k_{\tau}=1$. We find 
\begin{align}
\label{case1}
(k''_{\Frob^{s-1} \circ \tau}, \dots, k''_{\tau}, k''_{\Frob^{-1} \circ \tau})= 
\begin{cases} (0,p, \dots,p, p), & \text{if } k_{\Frob^{-1}\circ \tau}=1, \Frob^{-1}\circ \tau  \in M,  \\
(0,p, \dots,p, p+1), & \text{if }  k_{\Frob^{-1}\circ \tau}=1, \Frob^{-1}\circ \tau \not \in M \\
 (0,p, \dots,p, p+1), & \text{if } k_{\Frob^{-1}\circ \tau}=2, \ftau{-1} \not \in \tilde{M} \\
 (0,p, \dots,p, p+2), & \text{if } k_{\Frob^{-1}\circ \tau}=2, \ftau{-1} \in \tilde{M} \\
\end{cases}
\end{align}
with possibly no instances of $p$ in the last three cases, and it follows from Lemma \ref{HasDiv}. 

We treat (ii) and (iii) simultaneously. In both cases there exist $s \geq 1$ and $t \geq 0$ such that $\Frob^{s+t} \circ \tau \in \tilde{M}$ and such that
\begin{align*}
k_{\Frob^{s+t-1} \circ \tau}=\dots=k_{\Frob^s \circ \tau}=1, k_{\Frob^{s-1} \circ \tau}= \dots = k_{\tau}=2,
\end{align*}
where $k_{\Frob^{-1} \circ \tau}=2$ with $\Frob^{-1} \circ \tau \in M$ or $k_{\Frob^{-1}\circ \tau}=1$ with $\Frob^{-1} \circ \tau \not \in M$.
Suppose first $t=0$, then we find 
\begin{align}
\label{eq: tis0}    
(k''_{\Frob^{s-1} \circ \tau}, \dots, k''_{\tau},k''_{\Frob^{-1} \circ \tau})=
\begin{cases}
(1,p+1, \dots, p+1,p+2), &   k_{\Frob^{-1} \circ \tau}=2, \Frob^{-1} \circ \tau \in \tilde{M}, \\
(1,p+1, \dots, p+1,p+1), &   \text{otherwise}.
\end{cases}
\end{align}
with no instances of $p+1$ if $s=1$ in the first case.
If $t \geq 1$, then we obtain 
\begin{align}
\label{eq: tis1} \nonumber
(k''_{\Frob^{s+t-1} \circ \tau},\dots, &k''_{\Frob^s \circ \tau}, k''_{\Frob^{s-1} \circ \tau} \dots, k''_{\tau},k''_{\Frob^{-1} \circ \tau})=\\
&\begin{cases}
(0,p, \dots,p, p+1,\dots, p+1,p+2), &   k_{\Frob^{-1} \circ \tau}=2, \Frob^{-1} \circ \tau \in \tilde{M}, \\
(0,p, \dots,p, p+1,\dots, p+1,p+1), &   \text{otherwise}.
\end{cases}
\end{align}
with no instances of $p$ if $t=1$.
We obtain divisibility by $\Ha_{\Frob^{s+t-1} \circ \tau} \cdots  \Ha_{\tau}$ by Lemma \ref{HasDiv} for any $s \geq 1$ and $t \geq 0$. 
\end{proof}

\begin{lemma}
\label{HasDiv}
Suppose $f$ is of weight $(k,l)$ with $k \in \Z^{\Sigma}_{\geq 0}$.
Suppose $k$ is such that for integers $s \geq 1, t \geq 0$ and an embedding $\tau \in \Sigma$ the tuple $(k_{\Frob^{s+t-1} \circ \tau}, \dots, k_{\tau}, k_{\Frob^{-1}\circ \tau})$ is as in \eqref{case1}, \eqref{eq: tis0} or \eqref{eq: tis1}. Then $f$ is divisible by $\Ha_{\Frob^{s+t-1} \circ \tau} \cdots \Ha_{\tau}$.
\end{lemma}
\begin{proof}
This follows from reiteratively applying Theorem \ref{hasseiso}: in each case it is clear that $f$ is divisible by $\Ha_{\Frob^{s+t-1}\circ \tau}$. The result follows by noting that after this division we can apply Theorem \ref{hasseiso} to obtain divisibility by $\Ha_{\Frob^{s+t-2} \circ \tau}$ and observing we can continue to do this until the desired $\Ha_{\tau}$. 
\end{proof}

We demonstrate the last part of the proof of Theorem \ref{mainthm} in an example.

\begin{example}
Let $F$ be an extension of degree eight in which $p$ is inert. Fix $\tau_0 \in \Sigma$ and suppose $(k,0) \in \Z_{\geq 1}^{\Sigma} \times \Z^{\Sigma}$ is such that \[k=(k_{\tau_0},k_{\tau_1},k_{\tau_2},k_{\tau_3},k_{\tau_4},k_{\tau_5},k_{\tau_6},k_{\tau_7})=(1,1,k_2,2,2,1,2,2),
\]
with $k_2 \geq 3$. We find $\tilde{M}=\{\tau_2,\tau_4,\tau_7\}$ and $M\setminus\tilde{M}=\{\tau_0,\tau_3,\tau_5,\tau_6\}$. We have 
\begin{align*}
k'=&(p,p+1,k_2-1,p+1,p+1,p,p+1,p+1), \\
k^{\mu_1}=&(p,p+1,k_2+1,p+1,p+1,p,p+1,p+1), \\
k^{\mu_2}=&(p,p+1,k_2-1,p+1,p+3,p,p+1,p+1), \\
k^{\mu_3}=&(p,p+1,k_2-1,p+1,p+1,p,p+1,p+3). 
\end{align*}
By the above matching procedure we find $k''=(0,p+1,k_2,1,p+2,0,p+1,p+2)$. By Theorem \ref{hasseiso} we obtain divisibility by $\Ha_{\tau_0}, \Ha_{\tau_3}$ and $\Ha_{\tau_5}$. After division by these invariants we obtain a form of weight $(1,1,k_2,2,2,1,1,p+2)$. We apply Theorem \ref{hasseiso} again to obtain divisibility by $\Ha_{\tau_6}$ and we obtain a form of weight $(1,1,k_2,2,2,1,2,2)$.
\end{example}

\begin{remark}
Note that one could also hope to simplify Theorem \ref{mainthm} so that one just needs two regular weights, in particular since this is possible for the local Galois-theoretic analogue (as in \cite[Theorem 1.1]{hwcrys} and see also \cite{kls}). We now explain why this does not work here, or at least not with the same proof method. In the two weight version we would keep the regular weights $(k',l')$ but replace all the weights of the form $(k^{\mu},l^{\mu})$ with the weight $(k^{\theta},l^{\theta})$, which is the weight of the Hilbert modular form
 \[
 \Theta_{\tilde{M}} \left( f \cdot \prod_{\tau \in {M} \setminus \tilde{M}} \Ha_{\tau} \right),
 \]
 where $\Theta_{\tilde{M}}$ is the composition of $\Theta_{\mu}$ for all $\mu \in \tilde{M}$ (note the order does not matter due to commutativity as per \cite[Corollary 9.4.2]{DS}).  This is the global version of the weights $(k^{\theta},l^{\theta})$ in \cite[Theorem 1.1]{hwcrys}.
We now show by example why we cannot simplify Theorem \ref{mainthm} to a two weight version with our current proof strategy.

Let $F$ be a quartic extension with $p$ inert. Fix $\tau_0 \in \Sigma$. Let $(k,0) \in \Z_{\geq 1}^{\Sigma} \times \Z^{\Sigma}$ be such that $k=(k_{\tau_0},k_{\tau_1},k_{\tau_2},k_{\tau_3})=(k_0,1,k_2,1)$ with $k_0,k_2 \geq 3$. Then 
\begin{align*}
(k',l') &=((k_0-1,p+1,k_2-1,p+1),(0,0,0,0)), \text{ and } \\
(k^{\theta},l^{\theta})&=((k_0+1,p+1,k_2+1,p+1),(-1,0,-1,0)).
\end{align*}
Following the same strategy as in the proof above, we would obtain two modular forms $f'$ and $f^{\theta}$ of these weights and we would obtain an equality
\[
\Theta_{\tau_0} (\Theta_{\tau_2}(f'))=\Ha_{\tau_0} \Ha_{\tau_2} f^{\theta},
\]
By Theorem \ref{thm822} we would obtain that $\Theta_{\tau_2}(f')$ is divisible by $\Ha_{\tau_0}$.
However, we also need divisibility by $\Ha_{\tau_2}$, which is not given to us by the same result. We would need a result of the following form:
\[
\Ha_{\tau}\, | \, \Theta_{\tau'}(\Theta_{\tau}(f)) \iff \Ha_{\tau} \, | \, \Theta_{\tau}(f) \text{ or } p | k_{\tau}.
\]
This would enable us to apply Theorem \ref{thm822} again, and then we can finish the proof as above.
\end{remark}

\section{Main result}
\label{mainresultsection}
In this section we state and prove our main result. We show that for partial weight one weights geometric modularity of weight $(k,l)$ is equivalent to crystalline liftability of $\rho|_{G_{F_v}}$ with weights $(k_{\tau},l_{\tau})_{\tau \in \Sigma_v}$ for all $v|p$, under some hypotheses.

As explained in the introduction our proof will use known and conjectural relations between crystalline liftability and algebraic modularity on the one hand, and algebraic modularity and geometric modularity on the other hand. In particular, we use results of \cite{GLS15} to relate the existence of certain crystalline lifts with small weights to algebraic modularity, and vice versa. Recall algebraic modularity and geometric modularity are conjectured to be equivalent for regular weights, below we summarise what is known due to Diamond--Sasaki and Yang.

\begin{theorem}
\label{equivalence}
Let $\rho: G_F \to \GL_2(\Fpb)$ be a Galois representation that is irreducible, continuous and totally odd. Let $(k,l) \in \Z_{\geq 2}^{\Sigma} \times \Z^{\Sigma}$ be a weight with $k \in \Xi_{\min}$. If $\rho$ is algebraically modular of weight $(k,l)$ then $\rho$ is geometrically modular of weight $(k,l)$. If $\rho$ is geometrically modular of weight $(k,l)$, then $\rho$ is algebraically modular of weight $(k,l)$  in the following situations:
\begin{enumerate}
    \item $F$ is real quadratic, $p\geq 5$ inert in $F$ and if $p=5$ we have that $\rho|_{G_{F(\zeta_p)}}$ does not have projective image isomorphic to $A_5 \cong \PSL_2(\F_5)$,
    \item $p \geq \min \{5,[F:\Q]\}$ totally split in $F$ and if $p=5$ we have that $\rho|_{G_{F(\zeta_p)}}$ does not have projective image isomorphic to $A_5 \cong \PSL_2(\F_5)$,
    \item $p(k_{\tau}-2)>k_{\Frob^{-1} \circ \tau}-2$ for all $\tau \in \Sigma$.
\end{enumerate}
\end{theorem}
\begin{proof}
The direction from algebraic to geometric is \cite[Theorem 6.5.5]{DSram}. The other direction is due to Yang, see Theorem 4.14, Theorem 4.28 and Theorem 4.10 in \cite{yang}. 
\end{proof}

We should note that \cite{DS} already showed that if $\rho$ is geometrically modular of some weight, then $\rho$ is algebraically modular of \emph{some} weight \cite[Proposition 7.5.4]{DS}.

As discussed previously, we use a local analogue of Theorem \ref{mainthm} in our proof. The analogue proves that for any local representation, or indeed the restriction of a global representation $\rho|_{G_{F_v}}$ for $v|p$, crystalline liftability of small irregular weights can be predicted in terms of regular weights and vice versa \cite[Theorem 1.2]{hwcrys}. We use this result together with Theorem \ref{mainthm} to prove our main result.

\begin{theorem}
\label{locglobviceversa}
Let $F$ be a totally real field in which $p$ is unramified. Suppose that $\rho:G_F \to \GL_2(\Fpb)$ is irreducible and modular. Assume $\rho|_{G_{F(\zeta_p)}}$ is irreducible and if $p=5$, that $\rho|_{G_{F(\zeta_p)}}$ does not have projective image isomorphic to $A_5 \cong \PSL_2(\F_5)$.  Assume that for any weight $(k,l)$ with $k \in \Z^{\Sigma}_{\geq 2} \cap \Xi_{\min}$, if $\rho$ is geometrically modular of this weight it will also be algebraically modular of this weight. Let $(k,0) \in \Z_{\geq 1}^{\Sigma} \times \Z^{\Sigma}$ be an irregular weight such that
\begin{enumerate}
    \item $1 \leq k_{\tau} \leq p$ for all $\tau \in \Sigma$,
    \item there is no $\tau \in \Sigma$ with $k_{\tau}=2$ and $k_{\Frob^{-1} \circ \tau}=1$,
    \item there is no $v \mid p$ such that $k_{\tau}=1$ for all $\tau \in \Sigma_v$,
    \item if there is $v \mid p$ such that $k_{\tau} \geq 2$ for all $\tau \in \Sigma_v$, then $\rho|_{G_{F_v}}$ is not of the form  $\mtwosmall{\chi_1}{\ast}{0}{\chi_2}$ where $\chi_1$ is unramified and $\chi_2|_{I_{F_v}}=\prod_{\tau \in \Sigma_v} \omega^{1-k_{\tau}}_{\tau}$.
\end{enumerate}

Then $\rho$ is geometrically modular of weight $(k,0)$ if and only if $\rho|_{G_{F_v}}$ has a crystalline lift of weight $(k_{\tau},0)_{\tau \in \Sigma_v}$ for all $v \mid p$.
\end{theorem}

Before we give the proof, we provide the reasons behind conditions (1)--(4). The first two of these come from the $p$-adic Hodge theory results in \cite{hwcrys}, the methods therein come from \cite{GLS14} which requires the weights to be small. Condition (2) could potentially be removed, see \cite[Remark 4.13]{hwcrys}. The third condition is a condition therein as well as in this paper's Theorem \ref{mainthm}. The final condition makes sure we can apply Theorem \ref{mainthm}, i.e. forces that condition (3) in that theorem is satisfied in cases not covered by the auxiliary Lemma \ref{impliedshape}. Note if there exists $v|p$ as in (4) above, locally the weight $(k,0)$ will be regular at $v$ and we find that $(k_{\tau},l_{\tau})_{\tau \in \Sigma_v}=(k'_{\tau},l'_{\tau})_{\tau \in \Sigma_v}= (k^{\mu}_{\tau},l^{\mu}_{\tau})_{\tau \in \Sigma_v}$. We finally remark that one could explore removing this condition by choosing the relevant forms to be ordinary at $v$ with prescribed eigenvalue, for any place $v$ that fails condition (4), but we do not pursue this in the present article.

\begin{proof}
Suppose first that for each $v \mid p$, $\rho|_{G_{F_v}}$ has a crystalline lift of weight $(k_{\tau},0)_{\tau \in \Sigma_v}$. Then by Theorem 1.2 of \cite{hwcrys}, we find for each $v \mid p$ that $\rho|_{G_{F_v}}$ has crystalline lifts of weight $(k'_{\tau},l'_{\tau})_{\tau \in \Sigma_v}$ and of weights $(k^{\mu}_{\tau},l^{\mu}_{\tau})_{\tau \in \Sigma_v}$ for all $\mu \in \tilde{M}$. By Theorem A of \cite{GLS15}, we find that $\rho$ is algebraically modular of weights $(k',l')$ and $(k^{\mu},l^{\mu})$ for each $\mu \in \tilde{M}$, since $k'_{\tau} \leq p+1$ and $k^{\mu}_{\tau} \leq p+1$ for all $\tau \in \Sigma$ and all $\mu \in \tilde{M}$. By Theorem \ref{equivalence}, $\rho$ is also geometrically modular of these weights. Using Lemma \ref{impliedshape} and (4) to apply Theorem \ref{mainthm}, we then find that $\rho$ is geometrically modular of weight $(k,0)$.

Now suppose $\rho$ is geometrically modular of weight $(k,0)$, then by Theorem \ref{mainthm} it is also geometrically modular of the weights $(k',l')$ and $(k^{\mu},l^{\mu})$ for each $\mu \in \tilde{M}$. By our assumptions, it is then also algebraically modular of these weights. Again by Theorem A of \cite{GLS15}, we find that for each $v \mid p$, $\rho|_{G_{F_v}}$ has a crystalline lift of weights $(k'_{\tau},l'_{\tau})_{\tau \in \Sigma_v}$ and of weights $(k^{\mu}_{\tau},l^{\mu}_{\tau})_{\tau \in \Sigma_v}$ for all $\mu \in \tilde{M}$. By Theorem 1.2 of \cite{hwcrys}, indeed we find that $\rho|_{G_{F_v}}$ has a crystalline lift of weight $(k_{\tau},0)_{\tau \in \Sigma_v}$ for all $v \mid p$.
\end{proof}

The direction from crystalline liftability to geometric modularity in the above result is significantly easier owing to the progress listed in Theorem \ref{equivalence}. This is also apparent in \cite{DS}, by way of Theorem 11.4.3. However, by Theorem \ref{equivalence} we know the implication from geometric modularity to algebraic modularity for all weights such that $p(k_{\tau}-2)>k_{\Frob^{-1} \circ \tau}-2$. In our proof we only need this result for the weights $(k',l')$ and $(k^{\mu},l^{\mu})$ for each $\mu \in \tilde{M}$, and there will be many $(k,0)$ such that this holds.

We finish with the proof of an auxiliary lemma enabling us to apply Theorem \ref{mainthm} in the proof of Theorem \ref{locglobviceversa}.
\begin{lemma}
\label{impliedshape}
Let $(k,0)$ be a small irregular weight as in Theorem \ref{locglobviceversa}. Let $v|p$. Suppose $\rho|_{G_{F_v}}: G_{F_v} \to \GL_2(\Fpb)$ has a crystalline lift of weights $(k'_{\tau},l'_{\tau})_{\tau \in \Sigma_v}$ and of weights $(k^{\mu}_{\tau},l^{\mu}_{\tau})_{\tau \in \Sigma_v}$ for all $\mu \in \tilde{M}$ as in Section \ref{weightcombinatorics}. Let $\mu \in \Sigma$ be such that $\mu \in \Sigma_v \cap \tilde{M}$. Then $\rho|_{G_{F_v}}$ is not of the form 
\[\mtwo{\chi_1}{\ast}{0}{\chi_2},\]
where 
\[
\chi_1|_{I_{F_v}}=\prod_{\tau \in \Sigma_v} \omega^{-l_{\tau}^{\mu}}_{\tau}\text{ and } \chi_2|_{I_{F_v}}=\prod_{\tau \in \Sigma_v} \omega^{1-k_{\tau}^{\mu}-l_{\tau}^{\mu}}_{\tau}.
\]
\end{lemma}

\begin{proof}
Fix an element $\mu \in \tilde{M}$ such that $\mu \in \Sigma_v$ and suppose that $\rho|_{G_{F_v}}$ is of the form above. Since our Hodge--Tate weights are small and distinct, we can use \cite{GLS14} to describe the restriction to inertia of the local representations (noting the opposite $p$-adic Hodge theory conventions). We will use this to obtain a contradiction.

For each $\tau \in \Sigma_v$ we write:
\[ 
s^{\mu}_{\tau}=-l^{\mu}_{\tau} \text{ and } t^{\mu}_{\tau}=    1-k^{\mu}_{\tau}-l^{\mu}_{\tau}, 
\]
so that 
\[
s^{\mu}_{\tau}=
\begin{cases} 
1, & \tau=\mu, \\
0, & \tau \neq \mu,
\end{cases}
\text{   and } t^{\mu}_{\tau}=
\begin{cases}
2-k^{\mu}_{\tau},  & \tau=\mu, \\
1-k^{\mu}_{\tau}, & \tau \neq \mu.
\end{cases}
\]
 For a subset $J \subset \Sigma_v$, we let
\[
s'_{\tau}=\begin{cases} 
1-k'_{\tau},& \tau \in J, \\
0,& \tau \not\in J,
\end{cases}
\text{ and } 
t'_{\tau}=\begin{cases} 
0, & \tau \in J, \\
1-k'_{\tau}, & \tau \not \in J, 
\end{cases}
\]
(note $l'_{\tau}=0$ for all $\tau$). 

We label the embeddings, but now choosing $\tau_0=\mu$. We write $s'_{i}$ for $s'_{\tau_i}$ and $s^{\mu}_i$ for $s^{\mu}_{\tau_i}$. Using \cite[Corollary 7.11]{GLS14}, given that $\rho|_{G_{F_v}}$ has lifts of both weights $(k'_{\tau},l'_{\tau})_{\tau \in \Sigma_v}$ and of weights $(k^{\mu}_{\tau},l^{\mu}_{\tau})_{\tau \in \Sigma_v}$ we obtain the following congruence
\[
\sum_{i=0}^{f-1} s'_i p^{f-i-1} \equiv \sum_{i=0}^{f-1} s^{\mu}_i p^{f-i-1} \mod p^f-1. 
\]
for some subset $J$ of $\Sigma_v$.
Since $s^{\mu}_i - s'_i \in [-p,p]$, we can use \cite[Lemma 7.1]{GLS14} to check whether this congruence is satisfied. Since $p$ is odd and we must have $s^{\mu}_i - s'_i \neq p-1$ for some $i$, the lemma implies the sequence $(s^{\mu}_i-s'_i)_i$ must consist of strings of $(0,\dots,0)$ and $\pm (-1,p-1,\dots,p-1,p)$ (with possibly no entries of $p-1$).

We find that
\[
s_0^{\mu}-s_0' = \begin{cases}
    1-(1-k_0')=1-(2-k_0)=k_0-1, & 0 \in J, \\
    1-0=1, & 0  \not \in J,
\end{cases}
\]
but we have
\[
s_1^{\mu}-s_1' = \begin{cases}
    -(1-k_1')=1-(1-p)=p, & 1  \in J \text{ and } 1 \not \in M,  \\
    -(1-k_1')=1-(1-(p-1))=p-1, & 1 \in J \text{ and } 1 \in M,\\
    0, & 1  \not \in J.
\end{cases}
\]
We cannot have $0 \not \in J$ as then we would have that $(s_0^{\mu}-s'_0,s_1^{\mu}-s'_1) \in (1,p-1),(1,p)$ or $(1,0)$ which gives a contradiction. So assume $0 \in J$ so that $s_0^{\mu}-s_0'=k_0-1$. Note by our assumptions $k_0 \neq 2$, so this implies $k_0=p$. This gives $s_0^{\mu}-s_0'=p-1$, which means we need $s_i^{\mu}-s_i'=-1$ for some $i$. But we have that $s^{\mu}_i=0$ whenever $i \neq 0$. So we need $s_i'=1$ for some $i$. However, letting $J_0=\{\tau \in \Sigma \,|\, k_\tau=1\}$, we have
\[
s_i'= 
\begin{cases}
    0, & i \not \in J, \\
    1-p, & i \in J, i \in J_0, i \in M, \\
    -p, & i \in J, i \in J_0, i \not \in M, \\
    1-k_i, & i \in J, i \not \in J_0, i \not \in M, \\
    2-k_i, & i \in J, i \not \in J_0, i \in M, 
\end{cases}
\]
and since $k_i \geq 2$ for all $i \not \in J_0$ we find that $s_i' \neq 1$ for all $i$, so we again obtain a contradiction.
\end{proof}

\subsection*{Acknowledgements}
I would like to thank Fred Diamond for suggesting the problem and for his support and suggestions. I would also like to thank both Fred Diamond and James Newton for comments on earlier versions of this paper. As the reader will have noticed, this paper builds on the work of Diamond--Sasaki and it is a pleasure to thank them both for conversations about their work.

Earlier parts of this work were supported by the Herchel Smith Postdoctoral Fellowship Fund and by the Engineering and Physical Sciences Research Council, through the EPSRC Centre for Doctoral Training in Geometry and Number Theory~[EP/L015234/1] (the London School of Geometry and Number Theory) at University College London and through the grant `Modular representation theory, Hilbert modular forms and the geometric Breuil-Mézard conjecture'~[EP/W001683/1].


\bibliographystyle{alpha}

\bibliography{references}

\maketitle

\end{document}